\documentclass{icmart}

\contact[viviane.baladi@ens.fr]{UMR 8553 of CNRS, 
DMA, \'Ecole Normale Sup\'erieure,  75005 Paris, FRANCE}

\numberwithin{equation}{section}

\newtheorem{theorem}{Theorem}[section]

\theoremstyle{definition}

\newcommand{\integer}{{\mathbb {Z}}}
\newcommand{\complex}{{\mathbb {C}}}
\newcommand{\real}{{\mathbb {R}}}
\newcommand{\TT}{{\mathbb {T}}}
\newcommand{\BB}{{\mathcal{B}}}
\newcommand{\FF}{{\mathcal{F}}}
\newcommand{\JJ}{{\mathcal{J}}}
\newcommand{\LL}{{\mathcal{L}}}
\newcommand{\MM}{{\mathcal{M}}}
\newcommand{\RR}{{\mathcal{R}}}
\newcommand{\UU}{{\mathcal{U}}}
\newcommand{\VV}{{\mathcal{V}}}
\newcommand{\NT}{{\overset{NT}{\rightarrow}}}

\newcommand{\divv}{\operatorname{div}}
\newcommand{\gradd}{\operatorname{grad}}

\title[Linear Response, Or Else]{Linear Response, Or Else}

\author[Viviane BALADI]{Viviane BALADI
\thanks{The toy model in Section~\ref{toy} was presented at a
minicourse at the Dynamical Systems Days in Antofagasta, Chile,  December 2007.}}

\begin{document}

\begin{abstract}
Consider a smooth one-parameter family $t\mapsto f_t$ of dynamical systems $f_t$,
with $|t|<\epsilon$. Assume 
that for all $t$ (or for many $t$ close to $t=0$) 
the map $f_t$ admits a unique SRB invariant  probability
measure $\mu_t$. 
We say that \emph{linear response} holds
if $t\mapsto \mu_t$ is differentiable at $t=0$
(possibly in the sense of Whitney), and if its derivative can
be expressed as a  function of $f_0$, $\mu_0$, and $\partial_t f_t|_{t=0}$.
The goal of this note is to present to a general mathematical audience recent results and
open problems in the  theory of linear response for chaotic dynamical systems, possibly
with bifurcations.
\end{abstract}

\begin{classification}
Primary 37C40; Secondary 37D25, 37C30, 37E05.
\end{classification}

\begin{keywords}
Linear response. Transfer operator. Ruelle operator. Physical measure. SRB measure. Bifurcations.
Differentiable dynamical system. Unimodal maps. Expanding interval maps. 
Hyperbolic dynamical systems.
\end{keywords}

\maketitle

\section{Introduction}
A discrete-time dynamical system is a self-map
$f:M\to M$ on a space $M$. To any point $x\in M$ 
is then associated its (future)
orbit $\{ f^n(x)\mid n \in \integer_+\}$ where
$f^0(x)=x$, and $f^n(x)=f^{n-1}(f(x))$, for
$n\ge 1$, represents
the state of the system at time $n$, given the ``initial
condition'' $x$. (If $f$ is invertible, one can also
consider the past orbit $\{ f^{-n}(x)\mid n \in \integer_+\}$.)
In this text, we shall always assume that $M$
is a compact differentiable manifold (possibly with boundary),
with the Borel $\sigma$-algebra,
endowed with a Riemannian structure and thus normalised Lebesgue measure.
Many natural dynamical systems are
``chaotic'' (in particular, a small error in the initial
condition will grow exponentially with time) and best understood
via ergodic theory. The ergodic approach often starts with
finding a ``natural''
 invariant probability measure $\mu$ (a probability measure is
invariant if $\mu(f^{-1}(E))=\mu(E)$ for every Borel set).
Lebesgue measure is not always invariant, although there are
important exceptions such as the angle-doubling map
$x\mapsto 2x$ modulo $1$ on the circle,  hyperbolic linear toral automorphisms
such as the  ``cat map'' $A_0$ defined in \eqref{catmap} below, or  symplectic
diffeomorphisms. However, many interesting dynamical systems which do
not preserve Lebesgue admit a 
``physical'' invariant probability measure: The \emph{ergodic basin} of 
an $f$-invariant probability measure
$\mu$  is the set of those initial conditions for which time averages
converge to the space average for every continuous function
$\varphi :M \to \complex$, i.e., the set
$$
\{ x \in M \mid \lim_{n \to \infty} \frac{1}{n} \sum_{k=0}^{n-1} \varphi(f^k(x))
=\int \varphi\, d\mu\, , \,\, \forall \varphi \in C^0 \}.
$$
An invariant probability measure
$\mu$ is called \emph{physical} if 
its ergodic basin
has positive Lebesgue measure.

If $\mu$ is $f$-invariant and absolutely continuous with respect to
Lebesgue then, if it is in addition ergodic, 
it is a physical measure because of the Birkhoff ergodic theorem.
It was one of the breakthrough discoveries of the 60's, by
Anosov and others, that many natural dynamical systems (in particular
 smooth hyperbolic attractors) admit
finitely many physical measures, while in general they do \emph{not} admit any absolutely
continuous invariant measure.
Physical measures are sometimes called SRB
\footnote{The notions of SRB and physical measures
do not always coincide, see \cite{Yo}. In the present expository note,
we shall  ignore this fact.} measures after
Sinai, Ruelle, and Bowen, who studied them in
the sixties \cite{Yo}.

Instead of a single discrete-time dynamical system $f$, let us now
consider a one-parameter family $t\mapsto f_t$ of dynamical systems
on the same space $M$, where $t\in [-,\epsilon,\epsilon]$, for
$\epsilon >0$. We assume that the map $t\mapsto f_t$ is   ``smooth''
(i.e.,   $C^k$ for some $1<k\le \infty$),
taking a suitable topology in the image, e.g., that of
$C^\ell$ diffeomorphisms, or (piecewise) $C^\ell$ endomorphisms of $M$, for some
$\ell > 1$. We can view $f_t$ as a perturbation of
the dynamics $f:=f_0$. Let us \emph{assume} that there exists
a closed set $\Lambda$, containing $0$ as an accumulation
point, such that the map $f_t$ admits
a unique physical measure for every $t\in \Lambda$. (We shall give examples where this
assumption holds below.) The question we are interested in
is:  \emph{Does the map $t\mapsto \mu_t$
inherit any of the smoothness of $t\mapsto f_t$ at the
point $t=0$?} In particular, is $t\mapsto \mu_t$  differentiable at $t=0$
(possibly by requiring $k$ and $\ell$ large enough)?

 As such, the question is not well defined, because we must
be more precise regarding both the \emph{domain} $\Lambda$
and the \emph{range} $\{\mu_t \mid t\in \Lambda\}$ of the map $t \mapsto \mu_t$.
If $\Lambda$ contains a neighbourhood $U$ of $0$, then differentiability
is understood in the usual sense, and differentiability properties
usually hold throughout $U$.
However, if $\Lambda$ does not contain
\footnote{One could also decide to restrict $\Lambda$ even if it 
originally contains
a neighbourhood of $0$.} any neighbourhood of $0$,
``differentiability'' of $t \mapsto \mu_t$ on $\Lambda$ should be understood in the
sense of the Whitney extension theorem, as was pointed out by Ruelle \cite{R2}. In other words, the map
$t\mapsto \mu_t$ is called $C^m$ at $0\in \Lambda$ for a real number $m>0$
if this map admits a $C^m$ extension
 from $\Lambda$ to an open neighbourhood of $0$.
If $0\le m<1$ this is just  continuity or H\"older continuity
on a metric set.
For $m=1$, e.g.,  then ``$\mu_t$ is $C^1$ in the sense of Whitney
on $\Lambda$ at $t=0$'' means that there exists a continuous function
$\mu_s^{(1)}$, defined for $s\in \Lambda$, so that 
$$\mu_s =\mu_0  +s \mu^{(1)}_s +R_s\, , \, 
\mbox{with }R_s=o(|s|)
\, , \quad \forall s\in \Lambda\, .
$$
In order to give a precise meaning to
$=o(|s|)$,  we need to
be more specific regarding the topology used in the \emph{range.} 
Even if $\mu_t$ has a density with respect to Lebesgue,
the $L^1$ norm of this
density can be too strong to
get differentiability.  What is often suitable is a distributional norm, i.e., the
topology of the dual of $C^r$ for some $r \ge 0$ ($r=0$ corresponds
to viewing $\mu_t$ as a Radon measure). In other words, 
the question is the differentiability of
$$
t \mapsto \int \varphi \, d\mu_t\,  .
$$
where the ``observable'' $\varphi$ belongs to $C^r(M)$.
In some cases $(C^r(M))^*$ 
 can be replaced by a space of
anisotropic distributions (see \S\ref{STST}).

We emphasize that considering a strict subset $\Lambda_0\subset \Lambda$  containing $0$
as an accumulation point may change the class of Whitney-$C^m$ maps at $0$:
A given map $\mu_t$ defined on $\Lambda$ could be (Whitney) $C^m$ at
$0\in \Lambda_0$, but \emph{not} (Whitney) $C^m$ at $0\in \Lambda$.
It seems fair to take a ``large enough'' $\Lambda$,
for example by requiring $0$ to be a Lebesgue density point in $\Lambda$
(i.e., $\lim_{r \to 0} m(\Lambda \cap [-r,r])/(2r)=1$),
or at least $0$ \emph{not} to be a point of dispersion in $\Lambda$
(i.e., $\lim_{r \to 0} m(\Lambda \cap [-r,r])/(2r)>0$).

\smallskip
We shall focus on $0<m\le 1$.  
(Higher differentiability results, including formulas,
can be obtained  \cite{Ruru}
if one makes stronger smoothness assumptions on the individual dynamical
systems $x\mapsto f_t(x)$ and on the map 
$t \mapsto f_t$.)
If we can prove, under some assumptions on
the family $f_t$, on the set $\Lambda$, and on  $k$, $\ell$, and $r$, that
the map $t\mapsto \mu_t$ is differentiable at $0\in \Lambda$, then it
is natural to ask if there is a formula for
$$
\partial_t \int \varphi \, d\mu_t |_{t=0}
$$
in terms of $f_0$, $\mu_0$, $\varphi$, and the vector field $v_0:=\partial_t f_t|_{t=0}$.
If such a formula exists, it is called the \emph{linear response formula}
(it gives the response to first order of the system in terms
of the first order of the perturbation).
We shall assume that the perturbation takes place
in the image point, i.e., there exists vector fields $X_s$
so that
\begin{equation}
v_s:=\partial_t f_t|_{t=s}=X_s \circ f_s \, , \quad \forall s, t \in [-\epsilon,\epsilon] \, .
\end{equation}
(If each $f_s$ is invertible, the above is just a  definion of $X_s$.)
The mathematical study of linear response has been initiated
by Ruelle. In \S~\ref{STST}, we shall present his pioneering result
\cite{R0} 
on smooth hyperbolic systems (Axiom A attractors). Let us
just mention now the key \emph{linear response formula} he obtained in \cite{R0}
for smooth hyperbolic attractors $f_t$ and smoooth observables $\varphi$:
\begin{equation}\label{magic}
\partial_t \int \varphi\, \rho_t\, dx |_{t=0}=
\sum_{j=0}^\infty \int \langle X_0, \gradd (\varphi \circ f_0^j)\rangle  \,  d\mu_0\, ,
\end{equation}
where the sum is exponentially converging.
In   \cite{R4}, Ruelle had shown how to derive \eqref{magic}
from heuristic arguments, which  suggested to consider the following
\emph{susceptibility function} associated to $f_t$ and 
$\varphi$:
\begin{equation}\label{SF}
\Phi_t(z)=
\sum_{j=0}^\infty \int  z^j \langle X_0, \gradd (\varphi \circ f_0^j)\rangle \, d \mu_0 \, .
\end{equation}
Under very weak assumptions, the power series
$\Phi_t(z)$ (often denoted
$\Phi_t(e^{i\omega})$) has a nonzero radius of convergence. If the radius of convergence
is $\le 1$ and the series in the right-hand-side of
\eqref{magic} does not converge, Ruelle \cite[($**$)]{R5} suggested that the value at $z=1$
could sometimes be obtained by analytic continuation, possibly giving 
the linear response formula. However, caution is necessary,
as it was  discovered since then
  (see Section~\ref{smooth} below) that linear response fails \cite{BBS}
in cases where  a meromorphic continuation was known to
exist
 \cite{R2},  (see also the presentation of the results of \cite{BMS} in Section~\ref{pw}.)

Before we sketch the contents of this note, we make two 
simple but essential remarks on 
\eqref{magic}.
First note that the higher-dimensional version of the Leibniz expression $(X\rho)'=X'\rho+X \rho'$
reads 
$$\rho \divv X + \langle X,\gradd \rho\rangle\, .$$
Second, defining the transfer operator associated to 
an invertible
\footnote{See \eqref{endoop} for the noninvertible version.} dynamical system $f_t$ (acting, e.g., on
$L^\infty$ or $L^1$) by
$$
\LL_t \varphi(x)=\frac{\varphi (f_t^{-1}(x))}{|\det Df_t (f_t^{-1}(x))|},
$$
we have $\int \LL_t(\varphi)\, dx= \int \varphi\, dx$, for all $\varphi$
(since the dual of $\LL_t$ preserves Lebesgue
measure, this is the change of variable formula in
an integral). If the transfer operator
has a nonnegative fixed point $\LL_t \rho_t=\rho_t\in L^1$, 
then $\mu_t=\rho_t\, dx$ is an absolutely continuous
invariant probability measure for $f_t$ and thus (if ergodic) a physical measure.
In this case, if the eigenvalue $1$
for $\LL_t$ is simple and isolated, Ruelle's formula \eqref{magic} and integration by parts give, 
\begin{align}\nonumber
\partial_t \int \varphi\, \rho_t\, dx |_{t=0}&=
\sum_{j=0}^\infty \int \langle X_0, \gradd (\varphi \circ f_0^j)\rangle  \rho_0   \, dx\\
\nonumber &=
-\sum_{j=0}^\infty \int \varphi \circ f_0^j  (\rho_0 \divv X_0 + \langle X_0,\gradd \rho_0
\rangle) \, dx\\
\nonumber&=
-\sum_{j=0}^\infty \int \varphi  \LL_0^j ( \rho_0 \divv X_0 + \langle X_0,\gradd \rho_0
\rangle )\, dx\\
\label{magic'}&=
-\int \varphi  (1-\LL_0)^{-1} ( \rho_0 \divv X_0 + \langle X_0,\gradd \rho_0
\rangle )\, dx\, .
\end{align}
 Note that the residue of 
$(1-z\LL_0)^{-1} ( \rho_0 \divv X_0 + \langle X_0,\gradd \rho_0
\rangle )\, dx$  at $z=1$
vanishes, because Lebesgue measure
is the fixed point of $\LL_0^*$, and
the manifold is boundaryless,  so that 
$\int (\rho_0 \divv X_0 + \langle X_0,\gradd \rho_0
\rangle) dx =0$, by  integration by parts.
The ``metaformula'' \eqref{magic'} for linear response 
in the last line can be guessed  by applying perturbation theory
to the fixed point $\rho_t$ of the operator $\LL_t$.
We shall see  in \S~\ref{STST} instances where the above is  a
rigorous argument, \emph{even in cases where $\mu_t$ is not absolutely
continuous with respect to Lebesgue} (then, $\mu_t$ is a distribution,
enjoying smoothness along unstable directions), and
in Section ~\ref{OE} instances
where the computation above is invalid, even in cases where $\mu_t$ 
\emph{is in fact} absolutely
continuous with respect to Lebesgue.
We emphasize that the tricky point is that the resolvent $(1-z\LL_0)^{-1}$
is evaluated at an expression involving differentiation of $\rho_0$: While $\rho_0$
itself often belongs to a space on which $\LL_0$ has nice spectral properties, this
is not always true for its derivative.

The note is organised as follows: In \S~\ref{toy}, we give a complete proof of
linear response in the baby toy model of smooth locally expanding circle maps.
Section~\ref{LR} contains an account of two nontrivial occurrences of
linear response in chaotic dynamics: The breakthrough \cite{R0} of Ruelle for smooth
hyperbolic systems is presented in \S~\ref{STST}, while Dolgopyat's result \cite{Do}
in a (not necessarily structurally stable) partially hyperbolic case is stated in \S~\ref{DoLR}.
The next section, which contains both recent
results and open problems, is devoted to situations where linear response
is
violated: We consider first the toy model of piecewise  expanding
interval maps, presenting in \S~\ref{pw} our results 
\cite{BS1, BS1'} with Smania, and those with Marmi--Sauzin
\cite{BMS}.
Then, we focus on the -- more difficult --  smooth, nonuniformly expanding,
unimodal interval maps, discussing  in  \S~\ref{smooth} the work of Ruelle \cite{R3},
together with
our  work with Smania  \cite{BS2', BS2}, and our recent paper with Benedicks and
Schnellmann \cite{BBS}. Finally, \S~\ref{theproofs} contains a brief account
of the techniques of proofs in \cite{BBS}. 

The survey  published by Nonlinearity in 2008 \cite{B2}
contains a broad viewed account of the results, open problems, and conjectures at the time,
with an emphasis on the role played by critical points
(or more generally homoclinic tangencies) in the breakdown of
linear response. That survey is thus complementary to the present more
introductory  presentation. (In view of the page limitation for this contribution, we sometimes 
do not give fully explicit statements and definitions, the reader
is invited to consult the quoted references for clarification.)

We refer to Ruelle's
articles \cite{R4,R5,R6} for motivation, applications to physics, and more conjectures.
See also the interesting approach of Hairer and Majda \cite{HM}, including
references of applications to climate-change.
 In the present note, we do not discuss linear response for
continuous time dynamics \cite{R8, BL}, or for dynamical systems
in infinite dimensions (such as coupled map lattices \cite{JL1, JL2}).

\section{The toy model of expanding circle maps}
\label{toy}

In this section we present a proof of linear response in the
(baby) toy model of smooth expanding circle maps. The result and proof
are  well known (and simpler than the analogous
arguments in \cite{GL2, BS1}), but we are not aware of any reference.

Let $M=S^1$ be the unit circle, and let $f:S^1\to S^1$
be a $C^2$ map which is \emph{$\lambda$-locally expanding,} i.e., there exists
$\lambda >1$ so that $|f'(x)|\ge \lambda$ for all $x$.
It is known \cite{KS} that such an $f$ admits a unique absolutely continuous
invariant probability measure $\mu=\rho\, dx$. This measure  is mixing and therefore ergodic.
So a $C^2$ locally expanding map
$f$ admits a unique physical measure.  In fact, $\rho$ is $C^1$, and
it is everywhere strictly positive.
The transfer
operator
\footnote{The number of terms in the sum is a constant finite
integer $\ge 2$, the \emph{degree} of the map.}
\begin{equation}\label{endoop}
\LL \varphi(x)=\sum_{f_0(y)=x} \frac{\varphi(y)}{|f_0'(y)|}
\end{equation}
is  bounded on $C^1(S^1)$. It is known (see  \cite{B0}, e.g.,  for the relevant references
to Ruelle and others) that $\rho$ is a fixed point of $\LL$,
that the eigenvalue $1$ of $\LL$ (acting on $C^1(S^1)$) has algebraic
multiplicity equal to one, and that the rest of the spectrum of $\LL$
is contained in a disc of radius strictly smaller than one. 
(Thus, $\LL$ acting on $C^1(S^1)$ has a spectral gap.)
Note that the eigenvector of $\LL^*$ for the eigenvalue $1$
is just normalised Lebesgue measure (by the change of variable
formula).

Fix $\lambda >1$, and consider  a $C^2$ path $t\mapsto f_t$ for $t\in (-\epsilon, \epsilon)$, where
each $f_t$ is now $C^3$ and locally $\lambda$-expanding
(then, $\LL_t$ acts on $C^2$, and $\rho_t\in C^2$). Assume that 
$\|f_t-f_s\|_{C^3(S^1,S^1)}=O(|t-s|)$.
Then, using the
fact that $\LL_t$ (acting on $C^2(S^1)$ or
$C^1(S^1)$) satisfies the following \emph{Lasota--Yorke} (or Doeblin--Fortet) \footnote{What is essential here is the
compact embedding of $C^j$ -- the strong norm -- in $C^{j-1}$ --
the weak norm.} inequalities
\begin{equation}\label{LaYo}
\|\LL_t^k\varphi\|_{C^{j}} \le C\xi^k \|\varphi\|_{C^{j}} + C ^k \|\varphi\|_{C^{j-1}}\, ,
\quad \forall \varphi \, , \forall k \ge 1\, , j=1, 2\, , 
\end{equation}
(with uniform $0<\xi <1$ and $C\ge 1$),
together with \footnote{See Step 1 in the proof 
of Theorem~\ref{toythm} for a stronger claim.}
$$\|(\LL_t-\LL_0)\varphi\|_{C^1}= O(|t|) \|\varphi\|_{C^2}\, ,$$ 
one obtains
strong deterministic stability:

\begin{theorem}[Strong deterministic stability, \cite{BY}]\label{prehistory}
There exists $C>0$ so that 
$$\|\rho_t-\rho_s\|_{C^1}\le C|t-s|\, \, , \, 
\forall \, t, s \in (-\epsilon,
\epsilon)\, .
$$
In addition, for any $t$ there exists $\tau <1$, so that, 
for all $s$ close enough to $t$, the spectrum of
$\LL_s$, acting on $C^1(S^1)$ or $C^2(S^1)$, outside of the disc of radius $\tau$ consists 
exactly in
the simple eigenvalue $1$.
\end{theorem}

The above result implies that $t\mapsto \mu_t$ is Lipschitz, taking
the $C^1$ topology of the density $\rho_t$ of $\mu_t$ in the image.

Assume now  further (this does not reduce much generality)
that  $v_t=\partial_s f_s|_{s=t}$ can be written
as $v_t=X_t \circ f_t$ with $X_t\in C^2$. Then, we have
linear response:

\begin{theorem}[Linear response formula]\label{toythm}
Viewing $\rho_t\in C^2$ as a
$C^1$ function, the map $t \mapsto \rho_t$ is differentiable, and  we have
$$
\partial_s \rho_s |_{s=t}=
-(1- \LL_t)^{-1} ((X_t \rho_t)') \, ,
\quad \forall t \in (-\epsilon,
\epsilon)\, .
$$
\end{theorem}

Note that  $X_t \rho_t$ is $C^2$ by assumption. Since integration by parts
on the boundaryless manifold $S^1$ gives $\int (X_t \rho_t)'\, dx=0$,
the residue of the simple pole  at $z=1$ of the resolvent $(z-\LL_t)^{-1}$
(acting on $C^1(S^1)$) vanishes at $(X_t \rho_t)'$.

\smallskip

We now prove Theorem~\ref{toythm}, assuming Theorem~\ref{prehistory}:

\begin{proof}[Proof of Theorem \ref{toythm}]
The proof consists in three steps, to be proved at the end:

\item \textbf{Step 1:}
Considering $\LL_t$ as a bounded operator from
$C^2(S^1)$ to $C^1(S^1)$, we claim that the map $t\mapsto \LL_t$ is differentiable, and that,
 for every $t\in (-\epsilon,\epsilon)$, we have
$$
\MM_t (\varphi):= \partial_s \LL_s (\varphi)|_{s=t}=
- X_t' \LL_t(\varphi) - X_t \LL_t \biggl (\frac{\varphi'}{f'} \biggr ) +
X_t \LL_t \biggl (\frac{\varphi f''}{(f')^2} \biggr ) \, .
$$
(This step will use $v_t =X_t \circ f_t$.)

\item \textbf{Step 2:}
Let $\Pi_t (\varphi)= \rho_t \cdot \int \varphi\, dx$ be the rank one
projector for the eigenvalue $1$ of $\LL_t$ acting on $C^1(S^1)$.
Then, for every $t \in (-\epsilon,\epsilon)$, we have
$$
\partial_s \rho_s|_{s=t}=(1-\LL_t)^{-1} (1-\Pi_t) \MM_t (\rho_t) \, .
$$
(Note that $\rho_t\in C^2$, but $\MM_t$ is an operator from $C^2(S^1)$
to $C^1(S^1)$.)

\item \textbf{Step 3:}
For every $t\in (-\epsilon,\epsilon)$, we have 
$$
(1-\LL_t)^{-1} [(1-\Pi_t) \MM_t (\rho_t)]= -
(1- \LL_t)^{-1} ((X_t \rho_t)')  \, .
$$

\smallskip

Theorem~\ref{toythm} follows from putting together Steps 2 and 3. To conclude,
we 
justify the three steps:

\smallskip
\noindent \textbf{Proof of Step 1:}
We must show that the operators defined for
$s\ne t$ by 
$$\RR_{t,s}:=
\frac{\LL_t - \LL_s}{t-s} - \MM_t
$$
satisfy
$\lim_{s \to t} \|\RR_{t,s}\|_{C^2(S^1) \to C^1(S^1)}=0$.
We start by observing that the number of branches of $f_s$ 
(which is just its degree) does not depend
on $s$. So for any fixed $t$ and any $x$, each inverse branch for
$f_s^{-1}(x)$, for $s$ close enough to $t$, can be paired with a well-defined
nearby inverse
branch $f_t^{-1}(x)$. For two such paired branches, we get, since $\varphi\in C^2$,
each $f_s$ is $C^3$, and $t\mapsto f_t$ is $C^2$, that
\begin{align*}
&\frac{\varphi(f_t^{-1}(x))}{|f'_t(f_t^{-1}(x))|}-
\frac{\varphi(f_s^{-1}(x))}{|f'_s(f_s^{-1}(x))|}=O((t-s)^2)
-(t-s) X'_t(x) \frac{\varphi(f_t^{-1}(x))}{|f'_t(f_t^{-1}(x))|}\\
&\qquad\qquad\qquad
-(t-s)X_t(x) \biggl [
\frac{\varphi'(f^{-1}_t(x)}{f'_t(f^{-1}_t(x))|f'_t(f^{-1}_t(x))|}
-\frac{\varphi(f^{-1}_t(x) f''_t(f^{-1}_t(x))}{(f'_t(f^{-1}_t(x)))^2 |f'_t(f^{-1}_t(x))|}
\biggr ]\, .
\end{align*}

\smallskip
\noindent \textbf{Proof of Step 2:}
Fix $t$. By Theorem~\ref{prehistory}, we can find a positively
oriented closed  curve $\gamma$ in the complex plane so that, for any
$s$ close to $t$, the simple eigenvalue
$1$ of $\LL_s$ is contained in the domain bounded by $\gamma$, and
no other element of the spectrum of $\LL_s$ acting on $C^2(S^1)$ lies in this domain.
Step 2 then uses classical perturbation theory for isolated simple
eigenvalues of bounded linear operators on Banach spaces
(see \cite{Kat}, e.g., see also \cite{KeLi} for the use of similar
ideas to get spectral stability), which tells us that, for any $\varphi \in C^2$
so that $\Pi_s(\varphi)=\int \varphi\, dx=1$, we have
\begin{equation}\label{aswemay}
\rho_s =\frac{1}{2i\pi}
\oint_\gamma (z-\LL_s)^{-1} \varphi(z) \, dz \, .
\end{equation}
(We used that $\int \rho_s \, dx=1$ for all $s$ and
$\LL_s^*(dx)=dx$.)
Next, for $z\in \gamma$,
we have the identity
$$
(z-\LL_t)^{-1}- (z- \LL_s)^{-1}
=(z-\LL_t)^{-1}(\LL_t-\LL_s) (z- \LL_s)^{-1}\, ,
$$
where we view $(z- \LL_s)^{-1}$ as acting on $C^2(S^1)$, the difference
$(\LL_t-\LL_s)$ as an operator from $C^2(S^1)$ to
$C^1$, and  $(z-\LL_t)^{-1}$ as acting on $C^1(S^1)$.
Letting $s$ tend to $t$, and recalling Step 1, we have proved
$$
\partial_s (z-\LL_s)^{-1}|_{s=t}=
(z- \LL_t)^{-1} \MM_t (z- \LL_t)^{-1} \, .
$$
Finally, taking (as we may) $\varphi=\rho_t\in C^2$ in \eqref{aswemay},
\begin{align*}
\partial_s \rho_s|_{s=t} &=\frac{1}{2i\pi}
\oint_\gamma  (z- \LL_t)^{-1} \MM_t (z- \LL_t)^{-1}\rho_t(z) \, dz \\
&=\frac{1}{2i\pi}
\oint_\gamma  (z- \LL_t)^{-1}  \frac{\MM_t(\rho_t(z))}{z-1} \, dz
\, .
\end{align*}
An easy residue computation completes Step 2.

\smallskip
\noindent \textbf{Proof of Step 3:}
It suffices to show
$
\MM_t \rho_t - \Pi_t \MM_t \rho_t = -(X_t \rho_t)' 
$.
Step 1 implies
$$
\MM_t \rho_t = -X_t' \rho_t -X_t 
 \LL_t \biggl ( \frac{\rho_t'}{f_t'} - \frac{\rho_tf''_t}{(f'_t)^2}\biggr )  \, .
$$
Now we use that $\rho_t'=(\LL_t \rho_t)'\in C^1$ and
$$
(\LL_t \varphi)'(x)=
\sum_{f_t(y)=x} \frac{\varphi'(y)}{|f'_t(y)|}\frac{1}{f'_t(y)}
-\sum_{f_t(y)=x} \frac{\varphi(y)f''_t(y)}{|f'_t(y)|(f'_t(y))^2}\, , 
$$
to see that
$$
\LL_t \biggl ( \frac{\rho_t'}{f_t'} - \frac{\rho_tf''_t}{(f'_t)^2}\biggr ) 
=\rho_t' \, .
$$
We have shown that $\MM_t \rho_t= - (X_t \rho_t)'$, so that
$\int \MM_t \rho_t \, dx=0$ and $\Pi_t  \MM_t \rho_t=0$, ending the proof of Step 3, and thus of the
theorem.
\end{proof}

\section{Linear response}
\label{LR}

\subsection{Smooth hyperbolic dynamics (structural stability)}
\label{STST}

A $C^1$ diffeomorphism $f:M \to M$ is called \emph{Anosov} if there exist
a $Df$-invariant continuous splitting $TM=E^u\oplus E^s$ of the tangent bundle
and  constants $C>0$ and $\lambda >1$ so that, for any $x \in M$, all $n\ge 1$,
all $v\in E^s(x)$, and all $w\in E^u(x)$, 
\begin{equation}\label{Anosovv}
\| Df^n_x(v)\|\le C \lambda^{-n} \|v\|\, ,
\,\, \| Df^{-n}_x(w)\|\le C \lambda^{-n} \|w\| \, .
\end{equation}
Thus, Anosov diffeomorphisms are generalizations of
the linear hyperbolic map 
\begin{equation}\label{catmap}A_0=\left (\begin{matrix}1&1\\1&2 \end{matrix}\right )
\end{equation}
on the two-torus. Indeed
(we refer  to \cite{KH}, e.g., for the basics of hyperbolic dynamics), 
a small smooth perturbation of $A_0$ is an Anosov
diffeomorphism. 
Anosov diffeomorphisms $f$ admit (finitely many) SRB
measures as soon as they are $C^{1+\epsilon}$, and the SRB
measure is unique if the diffeomorphism is transitive.
\footnote{Transitivity is automatic if $f$ is volume
preserving. It is conjectured that all Anosov diffeomorphisms on
connected compact manifolds are transitive.}
For \emph{Axiom A diffeomorphisms,}  hyperbolicity
(i.e., the existence of the continuous splitting $E^u \oplus E^s$)
is assumed only at $T_xM$ for points
$x$ in the nonwandering set $\Omega$; in addition,
periodic orbits are assumed
to be dense in $\Omega$. Smale's horseshoe is a famous Axiom A diffeomorphism, but SRB
measures exist in general
only for Axiom A \emph{attractors,} such as the solenoid.
(Anosov diffeomorphisms are special cases of Axiom A
attractors.)
An important property of Axiom A diffeomorphisms is \emph{structural
stability:} If $f_0$ is an Axiom A attractor, and $f_t$ is close
to $f_0$ (in the $C^1$ topology), then $f_t$ is also Axiom A, and, in addition
$f_0$ is topologically conjugated to $f_t$, i.e., there is a one-parameter
family
\footnote{The map
$t\mapsto h_t$ is smooth and its derivative $\alpha_t$ solves
the twisted cohomological equation \eqref{TCE},
see also \cite{B2} and references therein.} of homeomorphisms $h_t$ so that $f_t=h_t \circ f_0 \circ h_t^{-1}$.

Linear response holds for smooth hyperbolic systems: After pioneering results
of de la Llave et al.\,\,\cite{LMM} and Katok et al.\,\,\cite{KKPW},  Ruelle proved the
following landmark theorem (\cite{R0, R1}, see also \cite{Ji}):

\begin{theorem}[Linear response for smooth hyperbolic systems]\label{ruellethm}
Let $M$ be a compact Riemann manifold.
Let $t\mapsto f_t$ be a $C^3$ map from $(-\epsilon,\epsilon)$
to $C^3$ diffeomorphisms $f_t:M\to M$.
Assume that each $f_t$ is a topologically mixing Axiom A attractor, and
let $\mu_t$ be its unique SRB probability measure.
Then for any $\varphi \in C^2$, the map
$
t \mapsto \int \varphi\, d\mu_t
$
is differentiable on $(-\epsilon,\epsilon)$. In addition, setting
$X_t =\partial f_s |_{s=t}\circ f_t^{-1}$, we have
\begin{equation}\label{Ruf}
\partial_s \int \varphi \, d\mu_s |_{s=t}
= \sum_{j=0}^\infty \int \langle \gradd (\varphi \circ f_t^j), X_t \rangle
\, d\mu_t\, ,
\end{equation}
where the series converges (exponentially).
\end{theorem}

In this situation, one shows that the  susceptibility function \eqref{SF} is holomorphic
in a disc of radius strictly bigger than one.

Ruelle exploited symbolic dynamics in \cite{R0, R1}. For
a more modern approach, using anisotropic Banach spaces,
see the work of Gou\"ezel and Liverani
(\cite[Thm 2.8]{GL1} for Anosov, and \cite[Prop. 8.1]{GL2} for Axiom~A).
The modern approach is much simpler, since the transfer operators
$\LL_t$ of the diffeomorphisms $f_t$ all have a uniform spectral
gap on the same Banach space $\BB$ of anisotropic
distributions, which contains, not only the SRB measure
$\mu_t$, but also its ``derivative.''  The ``metaformula'' \eqref{magic'}
can then be easily justified rigorously.

\subsection{Mild bifurcations}
\label{DoLR}

In \S~4 we shall see examples where the breakdown of structural
stability (the presence of bifurcations in the family $f_t$) is mirrored by  a breakdown of linear response.
However, structural stability is \emph{not} necessary to
obtain linear response -- and neither is the  spectral gap
\footnote{See  the work of Hairer and Majda \cite{HM}.}
of the transfer operator $\LL_t$.  We briefly describe a
result of Dolgopyat \cite{Do} on a class of partially hyperbolic maps.
We consider
partially hyperbolic  diffeomorphisms $f:M\to M$
on a smooth compact manifold $M$, i.e.,   we assume
the tangent bundle  is decomposed into invariant
bundles $E^c\oplus E^u\oplus E^s$, where
$E^u$ and $E^s$ are both nontrivial and enjoy \eqref{Anosovv}.
A  partially hyperbolic diffeomorphism $f$ is  called
an \emph{Anosov element of a standard abelian Anosov action} if the
central bundle $E^c$ of $f$ is tangent to the orbits of a 
$C^\infty$ action $g_t$ of $\real^d$ so that
 $f g_t=g_t f$ (see \cite{KS1,KS2}). 
 Assume further that $f$ admits a unique physical (SRB)
 measure $\mu$, whose
 basin has total Lebesgue measure. The action is called \emph{rapidly mixing}
if there exists 
and a ($g_t$-admissible)
class of smooth functions $\FF$, and, for any $m\ge 1$, there exists $C \ge 1$
so that, for all subsets $S$ in a suitable class
of unstable leaves of $f$,  any $\varphi\in \FF$, and for any smooth probability
density $\psi$ on $S$, we have 
$$
|\int_S (\varphi\circ f^n)(x) \psi(x)\, dx-
\int \varphi\, d\mu| \le C \|\varphi\|_\FF \|\psi\| n^{-m}\, . 
$$
We refer to \cite{Do} for precise definitions of the objects
above and of $u$-Gibbs states, we just recall here
that SRB measures are $u$-Gibbs states. Dolgopyat's result follows:

\begin{theorem}[Linear response for rapidly mixing abelian Anosov actions \cite{Do}]
\label{DoTh}
Let $f$ be a $C^\infty$ Anosov element of a standard abelian Anosov action so that
$f$ has a unique SRB measure and is rapidly mixing. Then, for
any $C^\infty$ one-parameter family of diffeomorphisms $t\mapsto f_t$ 
through $f_0=f$, choosing for each $t$ a $u$-Gibbs state $\nu_t$ for $f_t$
(which can be the SRB measure if it exists), we have that 
$\int \varphi\, d\nu_t$ is differentiable at $t=0$ for any $\varphi\in C^\infty$,
and the linear
response formula \eqref{Ruf} holds. (See \cite[p. 405]{Do} for the linear response formula.)
\end{theorem}

\noindent Besides giving a new proof in the Anosov case, applications of Theorem~\ref{DoTh} include:
\begin{itemize}
\item
time-one maps $f$ of Anosov flows, which are generically rapidly mixing;
\item
toral extensions $f$ of Anosov diffeomorphisms $F$ defined by
$$
f(x,y)=(F(x), y+\omega(x))\, ,\quad x \in M\, ,\, y \in \TT^d\, ,
\,\omega \in C^\infty(M, \TT^d)\, ,
$$ 
which are generically rapidly mixing (under a diophantine condition).
\end{itemize}

It seems important here that structural stability may only break
down in the central direction. This allows Dolgopyat to use
rapid mixing to prove that
most orbits can be shadowed, a key feature of his argument.

\section{Or Else}
\label{OE}

The results stated in \S~3.1 gave at the time some
hope \cite{R2} that linear response could hold (at least
in the sense of Whitney) for a variety of nonuniformly
hyperbolic systems. In the present section we shall state some results obtained
since 2007 which indicate that the situation is not
so simple. We would like to mention that numerical experiments
and physical arguments already gave a hint that something
could go wrong (see  \cite{Er}, e.g., for fractal transport, see \cite{KHK}).

\subsection{Piecewise expanding interval maps}
\label{pw}

Piecewise expanding maps can be viewed as a toy model for the smooth
unimodal maps to be discussed in \S~\ref{smooth}.
The setting is the following: We let $I=[-1,1]$ be a compact interval, and consider continuous
maps $f :I \to I$ with $f(-1)=f(1)=-1$, and so that 
$f|_{[-1,0]}$ and $f|_{[0,1]}$ are $C^2$, with $\inf_{x \ne c} |f'(x)|\ge \lambda >1$.
Such a map is called a \emph{piecewise expanding unimodal map}
(for $\lambda$).
Lasota and Yorke   \cite{LY} proved in the 70's that such a map
posesses a unique absolutely continuous invariant probability
measure $\mu=\rho\, dx$, which is always ergodic. In fact, the density
$\rho$ is of bounded variation.
If $\mu$ is mixing,
we have exponential decay of correlations for smooth
observables, which can be proved by using the
spectral gap of the transfer operator $\LL_t$ defined by \eqref{endoop}  acting on the Banach space
$BV$ of functions of bounded variation, see e.g. \cite{B0}.
We set $c=c_0=0$, and we put
$c_k=f^k(c)$ for $k\ge 1$.

Consider now a $C^1$ path $t\mapsto f_t $, with 
each $f_t$ a piecewise expanding unimodal map. Assume in addition that
$f_0=f$ is
topologically mixing  on $[c_2,c_1]$ (then $\mu=\mu_0$ is mixing), that $c_1<1$, and that
$c$ is not a periodic point of $f_0$ (this implies that $f_0$
is stably mixing, i.e., small perturbations of $f_0$ remain mixing).
Then, applying \cite{LY},
each $f_t$ admits a unique SRB measure $\mu_t=\rho_t \, dx$ (and each transfer
operator $\LL_t$ has a spectral gap on $BV$, the corresponding estimates
are in fact uniform).
Keller \cite{Ke} proved that the map
$$
t \mapsto \rho_t \in L^1(dx)
$$
is H\"older for every  exponent $\eta <1$. In fact, Keller showed
\begin{equation}\label{Kelcond}
\|\rho_t - \rho_s\|_{L^1}\le C  |t-s| |\log |t-s|  |\, .
\end{equation}
>From now on, we assume that each $f_t$ is piecewise $C^3$, that
the map $t\mapsto f_t$ is $C^2$, and that $v=\partial_t f_t|_{t=0}= X \circ f$.
An example is given by taking the \emph{tent maps} 
\begin{equation}\label{tentt}
\begin{split}
f_t(x)&=a+t - (a+t+1)x \, ,\mbox{ if } x \in [0,1]\, ,\\
f_t(x)&=a+t +(a+t+1)x \, ,  \mbox{ if } x \in[-1,0]\, ,
\end{split}
\end{equation}
choosing $0<a<1$ so that
$0$ is not periodic for $f_a$ and so that $f_a$ is
mixing (note that
$X_0(x)=(a+1)^{-1}(x+1)$). Observe that structural stability is strongly violated
here: $f_t$ is topologically conjugated to $f_s$ only if $s=t$
\cite{dMvS}. In other words, the family $f_t$ of tent maps undergoes strong bifurcations.

 A piecewise expanding map is called \emph{Markov} if $c$
is preperiodic, that is, if there exists $j\ge 2$
so that $c_j$ is a periodic point: $f^p(c_j)=c_j$
for some $p\ge 1$. (In this case, one can show that the invariant density
is piecewise smooth, and the susceptibility function is meromorphic.)
A Markov map is mixing if its transition matrix is aperiodic,
stable mixing then allows to construct easily  mixing tent maps. 

It turns out that Keller's upper bound
\eqref{Kelcond} is optimal, \emph{linear response fails:}

\begin{theorem}[Mazzolena \cite{MM}, Baladi \cite{B1}]\label{counter}
There exist a Markov piecewise expanding interval map $f_0$,
a path $f_t$ through $f_0$,  with a $C^\infty$ observable
$\varphi$, a constant $C>0$, and a sequence $t_n \to 0$, so that
$$
|\int \varphi\, d\mu_{t_n} - \int \varphi\, d \mu_0|\ge C |t_n| |\log |t_n||\, ,
\quad \forall n\, .
$$
\end{theorem}

Setting $v=v_0=\partial_t f_t|_{t=0}$, and assuming $v=X\circ f$, we introduce
\begin{equation}\label{defhor}
\JJ(f,v)=\sum_{j=0}^\infty \frac{v(f^{j}(c))}{(f^j)'(c_1)} =
\sum_{j=0}^\infty \frac{X(f^{j}(c_1))}{(f^j)'(c_1)} \, .
\end{equation}

If $\JJ(f_0,v_0)=0$ (a codimension-one condition on the perturbation $v$ or $X$), 
we say that the path $f_t$ is \emph{horizontal} (at $t=0$).
This condition was first studied for smooth unimodal
maps \cite{Ts', ALM}.
 In the setting of piecewise expanding unimodal maps,
Smania and I proved the following result:

\begin{theorem}[Horizontality and tangency to the topological class \cite{BS1, BS1'}]
A  path $f_t$ is called tangent to the topological class of $f_0$
(at $t=0$) if there exist a path $\tilde f_t$ so that
$f_t-\tilde f_t =O(t^2)$ and homeomorphisms $h_t$
so that $\tilde f_t \circ h_t=h_t \circ f_0$.  Then:
\begin{itemize}
\item
The path $f_t$ is horizontal (at $t=0$) if and only if there is a continuous solution
$\alpha$ to the twisted cohomological equation
\begin{equation}\label{TCE}
v(x)=X \circ f(x) = \alpha \circ f (x)-f'(x) \alpha (x)\, ,\quad
x \ne c\, .
\end{equation}
\item
The path $f_t$ is horizontal (at $t=0$) if and only if it is tangent to the topological
class of $f_0$ (at $t=0$).
\end{itemize}
\end{theorem}

Note that the family of tent maps given in \eqref{tentt} is \emph{not} horizontal.

\smallskip 

We already mentioned that $\rho_t \in BV$. Any function $g$ of bounded variation
can be decomposed as two functions of bounded variation
$g=g^{sing}+g^{reg}$, where the regular component $g^{reg}$ is
a continuous function of bounded variation, while the singular component 
$g^{sing}$ is
an at most countable sum of jumps (i.e., Heaviside functions). In the
particular case of the invariant density $\rho_t$ of a piecewise
expanding unimodal map, we proved
\cite{B1} that $(\rho_t^{reg})'$ is of bounded variation,
while the jumps of $\rho_t^{sing}$ are located along the postcritical
orbit $c_k$, with exponentially decaying weights, so that $(\rho_t^{sing})'$
is an exponentially decaying sum of Dirac masses along the postcritical orbit.
\emph{The fact that the derivative of $\rho_0$ does not belong to
a space on which the transfer operator has a spectral gap is the
glitch which disrupts the spectral perturbation
mechanism described in Section~\ref{toy}} (in Section \ref{STST}  the derivative of the distribution corresponding to the SRB measure
\emph{did} belong to a good space of anisotropic distributions). Note also
that $\rho_0^{sing}$ is intimately related to the
postcritical orbit of $f_0$, which is itself connected to the
bifurcation structure of $f_t$ at $f_0$. (We refer also to \cite{B2}.)

\medskip 

Our main result with Smania on piecewise expanding maps  reads as follows:

\begin{theorem}[Horizontality and linear response \cite{BS1}]
\quad \phantom{
If the path $f_t$ is hor-}
\begin{itemize}
\item
If the path $f_t$ is horizontal (at $t=0$) then the map
$t\mapsto \mu_t\in C(I)^*$ is differentiable at $t=0$
(as a Radon measure), and we have the linear response formula:
\begin{equation}\label{LR0}
\partial_t \mu_t|_{t=0}=
- \alpha (\rho^{sing})' - (1- \LL_0)^{-1} (X' \rho^{sing} + (X \rho^{reg})') \, dx\, .
\end{equation}
\item
If the path $f_t$ is not horizontal (at $t=0$), then, if
in addition $|f'(c_-)|=|f'(c_+)|$ or $\inf_j d(f^j(c),c)>0$, we have:

If the postcritical orbit $\{c_k\}$ is not
\footnote{Generically the postcritical orbit is dense, see the references
to Bruin in \cite{Sc}.} dense in $[c_2, c_1]$, then
there exist $\varphi \in C^\infty$ and $K>0$
so that for any sequence $t_n \to 0$ so that the postcritical
orbit of each $f_{t_n}$ is infinite, 
\begin{equation}\label{starstar}
\left | \int \varphi \, d\mu_{t_n} - \int \varphi \, d\mu_0\right |
\ge K |t_n| |\log |t_n|| \, , \quad \forall n \, .
\end{equation}

If the postcritical orbit is dense in $[c_2,c_1]$, then there exist $\varphi \in C^\infty$
and sequences $t_n\to 0$ so that
\begin{equation}\label{starstarstar}
\lim_{n \to\infty} \frac{\left | \int \varphi \, d\mu_{t_n} - \int \varphi \, d\mu_0\right |}
{|t_n| } =\infty \, .
\end{equation}
\end{itemize}
\end{theorem}

We end this section with some of our results on the susceptibility function
(recall \eqref{SF})
$$
\Psi_\varphi(z)=
\sum_{j=0}^\infty \int z^j   (\partial_x (\varphi \circ f_0^j) (x))  X_0(x) \, d\mu_0(x)
$$
of piecewise expanding unimodal maps (for $\lambda>1$), the most recent of which were
obtained with Marmi and Sauzin (using work of Breuer and Simon
\cite{BrSi}):

\begin{theorem}[\cite{B1, BMS}] 
There exists a nonzero function $\UU(z)$, holomorphic in
$|z|> \lambda^{-1}$, and, for every non constant $\varphi\in C^0$
so that $\int \varphi\, d\mu_0=0$, there exists
a nonzero function $\VV_\varphi(z)$, holomorphic in
$|z|> \lambda^{-1}$, so that the following holds:
Put
$$\sigma_\varphi(z)=\sum_{j=0}^\infty \varphi(c_{j+1}) z^j$$
(this function is holomorphic in the open unit disc),
and set
$$
\Psi^{hol}(z)=- \int \varphi(x) (1-z \LL_0)^{-1}(X' \rho^{sing} + (X \rho^{reg})') (x)\, dx\, .
$$
Then:
\begin{itemize}
\item
There exists $\tau\in(0,1)$ so that $\Psi^{hol}(z)$ 
is holomorphic in the disc $|z|<\tau^{-1}$.
\item
The susceptibility
function satisfies
$$
\Psi_\varphi(z)=\sigma_\varphi(z) \UU(z) +\VV_\varphi(z)+\Psi^{hol}(z)\, ,
$$
where the function $\UU(z)$ vanishes at $z=1$ if and only if 
$\JJ(f,v)=0$, and in that case, we have
$$
\partial_t \int \varphi\, d\mu_t|_{t=0}= \VV_\varphi(1) + \Psi^{hol}(1)\, .
$$
\item
If $\{c_k\}$ is dense in $[c_2,c_1]$ and $\varphi\ne 0$, then
the unit circle is a (strong) natural boundary for $\sigma_\varphi(z)$
(and thus for $\Psi_\varphi(z)$).
If
\footnote{This assumption of Birkhoff
genericity of the postcritical orbit is
generic \cite{Sc}.} $\lim_{n \to \infty} \frac{1}{n} \sum_{k=1}^{n} \tilde \varphi(c_k)=
\int \tilde \varphi\, d\mu_0$ for every $\tilde \varphi\in C^0$, then
for every $\omega \in \real$ 
$$
\lim_{z \NT e^{i\omega}} (z-e^{i\omega})\sigma_\varphi(z)=0  \, ,
$$
where $z \NT e^{i\omega}$ means that $|z|<1$ tends to $e^{i\omega}$ nontangentially
(e.g., radially).
\end{itemize}
\end{theorem}

In particular, if the path $f_t$ is horizontal 
(at $t=0$) and the postcritical orbit is generic, then
$$
\partial_t \int \varphi\, d\mu_t|_{t=0}= 
\lim_{z \NT 1} \Psi_\varphi(z)\, .
$$

The law of the iterated logarithm (LIL), a property stronger  than Birkhoff genericity,
also holds generically for the postcritical orbit of piecewise expanding maps \cite{Sc'}.
If the postcritical orbit satisfies (an $e^{i\omega}$ twisted 
upper bound version of) the LIL, then more can be said
about $\sigma_\varphi$ and $\Psi_\varphi$, see \cite[Thm. 5]{BMS}.

Inspired by Breuer--Simon, we introduced in \cite{BMS} \emph{renacent right-limits,} a simple construction 
for candidates
for a generalised (Borel monogenic \cite{Bo}, e.g.) 
continuation outside of the unit disc of  power series
having the unit circle as a natural boundary. 
In the case of Poincar\'e simple pole series,
Sauzin and Tiozzo \cite{ST} showed that this construction gives the
 (unique) generalised continuation. However, for the susceptibility function
of piecewise expanding maps, there are \cite{BMS} uncountably many such candidates
(even in the horizontal case). This may indicate that there is 
no reasonable way to extend $\Phi_\varphi(z)$ outside of the unit
circle. The analogous problem is more delicate
for smooth unimodal maps discussed in \S~\ref{smooth} below,  mainly
because the natural boundary for the susceptibility function is expected to lie
strictly inside the open unit disc --- we
refer to  \cite{BMS} for open questions and conjectures.

\subsection{Smooth unimodal maps}
\label{smooth}

We now consider the more difficult case of \emph{differentiable} maps
$f:I\to I$, where $I=[0,1]$ is again a compact interval, and $c=1/2$
is now a critical point in the usual sense: $f'(c)=0$. The map $f$
is still assumed unimodal, with $f(-1)=f(1)=-1$, and $f'(x) >0$ for $-1\le x<c$, while
$f'(x)<0$ for $c<x\le 1$. We denote $c_k=f^k(c)$ for $k\ge 1$ as before.
For convenience, we assume that 
$f$ is topologically mixing and  $C^3$, with negative Schwarzian derivative
(see \cite{dMvS}).  Finally, we suppose that $f''(c)<0$.
Of course, $f$ is \emph{not} uniformly expanding since $f'(c)=0$.
One way to guarantee enough (nonuniform) expansion  is via the
Collet--Eckmann condition: The map $f$ is \emph{Collet--Eckmann} (CE)
if there exists $\lambda_c>1$ and $H_0\ge 1$ so that
$$
|(f^k)'(c_1)|\ge \lambda_c^k \, , \quad \forall k \ge H_0\, .
$$
If $f$ is CE, then it admits a (unique) absolutely continuous (SRB) invariant probability 
measure $\mu=\rho \, dx$
(which is ergodic).
We refer to \cite{dMvS} for more about  the CE condition,  noting here only
that the invariant density $\rho$ is not bounded in the current setting --- in fact, 
$\rho$ contains
a finite, or infinite exponentially decaying, sum of ``spikes''
$$\sqrt{|x-f^k(c)|}^{-1}$$ along the postcritical orbit. Thus,  $\rho\in L^p$
for all $1\le p<2$, but $\rho \notin L^2$.
If $f$ is CE and topologically mixing on $[c_2,c_1]$, then Keller and Nowicki \cite{keno}
and, independently, Young \cite{young92}, proved that a spectral gap holds for
a suitably defined transfer operator (acting on a ``tower''), giving exponential
decay of correlations.   

We consider again a $C^2$ path $t\mapsto f_t$, $t\in (-\epsilon, \epsilon)$,
say, of $C^4$ unimodal
maps as above, through $f=f_{t_0}$ (with $t_0$ not necessarily
equal to $0$) which will be assumed to
be (at least) CE. We let $v=v_{t_0}=\partial_t f_t|_{t=t_0}$ 
and assume that $v=X \circ f$.
Noting that $\JJ(f,v)$ from \eqref{defhor} 
is well defined because of the CE condition, we say that the path
$f_t$ is horizontal at $t=t_0$ if $\JJ(f,v)=0$.

The fully horizontal case (i.e., $\JJ(f_t,v_t)=0$ for all
$t$ in a neighbourhood of $t_0$) happens when
$f_t$ is topologically conjugated to $f_{t_0}$ for all $t$, so that $f_t$
stays in the topological class of $f_{t_0}$. Then, if
$f_{t_0}$ is Collet--Eckmann, all the $f_t$ are Collet--Eckmann 
(although it is not obvious from the definition,
the CE property is a topological invariant \cite{NP}) and
admit an SRB measure. In this fully horizontal case, viewing
$\rho_t$ as a distribution of sufficiently high order, first  Ruelle \cite{R3}
and then Smania and myself
\cite{BS2', BS2} obtained linear response,
with a linear response formula. (In \cite{BS2'}, we even obtain analyticity
of the SRB measure.) More precisely, Ruelle \cite{R3} considered
 the analytic case under the   Misiurewicz 
\footnote{Misiurewicz is nongeneric. It implies Collet--Eckmann.} assumption that
$\inf_k |f_{t_0}^k(c)-c|>0$;
  Smania and myself considered
on the one hand  \cite{BS2'}  a fully holomorphic setting (where the powerful
machinery of Ma\~n\'e--Sad--Sullivan \cite{MSS} applies), and on the other hand
 \cite{BS2}
a finitely differentiable setting  under a (generic) Benedicks--Carleson-type
assumption of topological slow recurrence.  The strategy in \cite{BS2} involves proving the
existence of a continuous solution $\alpha$ to the twisted cohomological
equation \eqref{TCE} if $f$ is Benedicks--Carleson and $X$ corresponds
to a horizontal path $f_t$.

Although the horizontal case is far from
trivial (in the present nonuniformly expanding setting, one of the
hurdles is to obtain uniform bounds on the constant
$\lambda_c(t)$ for  CE parameters $t$ close to $t_0$), it is much more interesting
to explore \emph{transversal paths} $t\mapsto f_t$ (undergoing topological bifurcations).
The archetypal such situation is given by the so-called \emph{logistic}
(or quadratic) family 
$$f_t(x)=tx(1-x)\, .
$$
A famous theorem of Jacobson  says that the
set of CE parameters in the logistic family has strictly positive
Lebesgue measure  (see \cite{dMvS}, e.g.). Since the set $\Lambda$
of CE parameters does not contain
any interval, regularity of the map $t\mapsto \mu_t$
for $t$ in $\Lambda$ can be considered only in the sense of Whitney.
Continuity of the map $t\mapsto \mu_t$, for $t$ ranging in some appropriate
subset of $\Lambda$ (and for the weak $*$ topology in the image)
was obtained by
Tsujii \cite{Ts} (see also  Rychlik--Sorets \cite{RySo}) in the 90's.

A map $f$ is called \emph{Misiurewicz--Thurston} if
there exist $j\ge 2$ and $p\ge 1$ 
so that $f^p(c_j)=c_j$ and $|(f^p)'(c_j)|>1$
(in other words, the critical
point is \emph{preperiodic,} towards a repelling periodic orbit,
this implies that the map has a finite  Markov partition). Clearly, Misiurewicz--Thurston implies
Misiurewicz and thus Collet--Eckmann. There are only countably many
Misiurewicz--Thurston parameters.

For the quadratic family, e.g.,
Thunberg proved \cite[Thm C]{thun}
that there are superstable parameters $s_{n}$ of periods $p_n$, with
$s_{n}\to t$, for a Collet--Eckmann parameter $t$,  
so that $\nu_{s_{n}}\to \nu$, where $\nu_{s_{n}}=\frac{1}{p_n}\sum_{k=0}^{p_n-1}
\delta_{f^k_{s_{n}}(c)}$,
and $\nu$ is the sum of atoms on a
repelling periodic orbit of $f_t$. Other sequences  $t_{n}\to t$ of superstable
parameters have the property that $\nu_{t_{n}}\to \mu_t$, the absolutely
continuous invariant measure of $f_t$. 
Starting from Thunberg's result, Dobbs and Todd  \cite{Dob} have 
constructed sequences of
 both renormalisable and non-renormalisable Collet--Eckmann 
  maps  $f_{t'_{n}}$, converging to a Collet--Eckmann map $f_t$,
but such that the SRB measures do not converge. 
Such counter-examples can  be constructed while requiring that $f_t$ and all maps
$f_{t'_{n}}$ are  Misiurewicz--Thurston.
These examples show that
continuity of the SRB measure cannot hold on the set of {\it all}
Collet--Eckmann (or even Misiurewicz--Thurston) parameters: Some
uniformity in the constants is needed (already when defining the ``appropriate
subsets'' of \cite{Ts}).

The main result of our joint work  \cite{BBS} with
Benedicks and Schnellmann (which also contains parallel
statements on more general transversal familes of smooth unimodal maps)
follows:

\begin{theorem}[H\"older continuity of the SRB measure in the logistic family \cite{BBS}]
\label{last}
Consider the quadratic family $f_t(x)=tx(1-x)$ on $I=[0,1]$, and let $\Lambda\subset (2,4]$ be the
 set of Collet--Eckmann parameters $t$.
\begin{itemize}
\item
There exists $\Delta \subset \Lambda$, of full Lebesgue measure in
$\Lambda$, so that for every $t_0\in \Delta$, and for every
$\Gamma > 4$, there exists $\Delta_{t_0} \subset
\Delta$, with $t_0$ a Lebesgue density point of $\Delta_{t_0}$, and there 
exists a constant $C$ so that, for any $\varphi\in C^{1/2}(I)$, for
any sequence $t_n \to t_0$,
so that $t_n\in \Delta_{t_0}$ for all $n$, we have
\begin{equation}\label{main0}
|\int \varphi(x) d\mu_{t_n}-\int \varphi(x) d\mu_{t_0}|
\le C \| \varphi\|_{C^{1/2}} |t_0-t_n|^{1/2} |\log |t_0-t_n||^\Gamma \, ,
\end{equation}
where $\| \varphi\|_{C^{1/2}}$ denotes the $1/2$-H\"older norm of $\varphi$.
\item
If $f_{t_0}$ is Misiurewicz--Thurston, then there exists $\varphi \in C^\infty$,
a constant $C >1$, and a sequence $t_n\to t_0$, with $t_n\in \Lambda$
for all $n$, so that
\begin{equation}\label{main}
\frac{1}{C} |t_n-t_0|^{1/2}\le |\int \varphi(x) d\mu_{t_n}-\int \varphi(x) d\mu_{t_0}|
\le C |t_n-t_0|^{1/2}\, .
\end{equation}
\end{itemize}
\end{theorem}

The exponent
$1/2$ appearing in the theorem is directly related to the
nondegeneracy assumption $f''(c)\ne 0$, which of course holds true
for the quadratic family. 
Note also that using a $C^\infty$ (instead of $C^{1/2}$)
observable does not seem to allow better upper bounds 
in \eqref{main0}. It is unclear if the logarithmic factor in \eqref{main0}
is an artefact of the proof or can be discarded. 

The proof of the  claim \eqref{main} of the  theorem  gives a sequence $t_n$ of Misiurewicz--Thurston
parameters, but the continuity result of Tsujii \cite{Ts} easily
yields sequences of non Misiurewicz--Thurston (but CE) parameters
$t_n$. We do \emph{not know} whether $t_0$ is a Lebesgue density point
of the set of sequences giving 
\eqref{main}. Note that in the toy model from \S~\ref{pw},
the first analogous construction of
counter-examples (Theorem~\ref{counter}) was  limited  to
a handful of preperiodic parameters
(sequences of maps having preperiodic critical points converging
to a map $f_{t_0}$ with a preperiodic critical point), while the currently known
set of examples (see \eqref{starstar} and \eqref{starstarstar}) are much
more general, although  not fully satisfactory yet.
One important open problem is to describe precisely the set of
sequences $t_n\to t_0$ giving rise to violation of linear response
for the generic piecewise expanding unimodal maps with dense postcritical orbits
in \eqref{starstarstar}. This may give useful insight for smooth unimodal maps,
both about the
largest possible set of sequences giving 
\eqref{main}, and about relaxing
the Misiurewicz--Thurston assumption on $f_{t_0}$. 
(Note however that there is a quantitative difference
with respect to the piecewise expanding case \cite{BS1}, where
the modulus of continuity in the transversal case was 
 $|\log |t-t_0|| |t-t_0|$, so that violation of linear response arose from
 the logarithmic factor alone.)

We suggested in \cite{BBS} the following  weakening
of the linear response problem: Consider a one-parameter family $f_t$ of
(say, smooth unimodal maps) through $f_{t_0}$ and, for each 
$\epsilon>0$, a random perturbation of $f_t$ with unique invariant
measure $\mu_{t}^\epsilon$ like in \cite{WX}, e.g. Then  for each positive
$\epsilon$, it should not be very difficult to
see that the map $t \to \mu_{t}^\epsilon$ is differentiable at $t_0$
(for essentially any topology in the image). 
Taking a weak topology in the
image,
like Radon measures, or distributions of positive order,
does the limit as $\epsilon \to 0$ of this derivative exist?
How is it related with the perturbation?
with the susceptibility function or some of
its generalised continuations (e.g. in the sense of \cite{BMS})?

\smallskip
 More open questions  are listed in
\cite{B2} and \cite{BS2, BBS}.
In particular, the results in \cite{BBS}  give hope that linear response
or its breakdown
(see  \cite{B2} and \cite{R7}) can be studied
for (the two-dimensional) H\'enon family,  which is transversal,
and where continuity
of the SRB measure in the weak-$*$ topology
was proved by Alves et al. \cite{ACF, ACF2} in the sense of Whitney 
on
suitable parameter sets.
In \cite[(17), (19)]{B2}, we also give candidates for the notion of horizontality 
for piecewise expanding maps in higher dimensions and piecewise
hyperbolic maps.

\smallskip
\subsection{About the proofs}
\label{theproofs}

The main tool in the proof of Theorem~\ref{last} is
a \emph{tower construction:} 
We wish to compare the SRB measure 
of $f_{t_0}$  to that of $f_t$ for  small $t-t_0$. 
Just like in \cite{BS2}, we use transfer operators
$\widehat \LL_t$ acting on towers, with a projection $\Pi_t$ from
the tower to $L^1(I)$ so that $\Pi_t \widehat \LL_t=\LL_t \Pi_t$,
where $\LL_t$ is the usual transfer operator, and $\Pi_t \hat \rho_t=\rho_t$
with $\mu_t=\rho_t\, dx$ (here, $\hat\rho_t$ is the fixed point
of $\widehat\LL_t$, and $\rho_t$ is the invariant density of $f_t$). In
\cite{BS2}, we adapted the tower  construction from \cite{BV}
(introduced in \cite{BV} to study random perturbations, for
which this version is better suited than the otherwise ubiquitous Young
towers \cite{YoYo}). This construction
allows in particular to work with
Banach spaces of continuous functions.  Another idea  imported from
\cite{BS2} is the use of  operators $\widehat \LL_{t,M}$
acting on truncated towers,
where the truncation level $M$ must be chosen carefully depending on $t-t_0$.
Roughly speaking, the idea is that   $f_t$ is comparable to $f_{t_0}$
for $M$ iterates (corresponding to the $M$ lowest levels
of the respective towers), this is the notion of an \emph{admissible pair} $(M,t)$.
Denoting by $\hat \rho_{t,M}$ the maximal eigenvector
of  $\widehat \LL_{t,M}$, the starting point for both  upper and lower bounds
is (like in \cite{BS2}) the decomposition 
\begin{align}\label{dec0}
\rho_t - \rho_{t_0}&=\bigl [\Pi_t (\hat \rho_t - \hat \rho_{t, M}) +
 \Pi_{t_0} (\hat \rho_{t_0,M}-\hat \rho_{t_0})\bigr ] \\
\nonumber &\qquad\quad+
[\Pi_t ( \hat \rho_{t,M} -\hat \rho_{t_0,M}  )]
+[ (\Pi_t -\Pi_{t_0}) (\hat \rho_{t_0,M})] \, ,
\end{align}
for admissible pairs.
The idea is then to get upper bounds on the first two terms by using
perturbation theory \`a la Keller--Liverani \cite{KeLi}, and to control
the last (dominant) term  by explicit computations on $\Pi_t -\Pi$
(which represents the ``spike displacement,'' i.e., the effect of the
replacement of $1/\sqrt{|x-f^k_{t_0}(c)|}$ by $1/\sqrt{|x-f^k_t(c)|}$ in the invariant
density).

We now move to the differences between \cite{BS2} and \cite{BBS}:
Using a tower with exponentially decaying levels
as in \cite{BV} or \cite{BS2} would  provide
at best an upper modulus of continuity $ |t-t_0|^{\eta}$ for $\eta <1/2$, and would
not yield any lower bound.
For this reason, we use instead ``fat towers''
with {\it polynomially decaying}
sizes in \cite{BBS}, working with polynomially recurrent maps. In order to construct the corresponding parameter set, we use recent results of Gao and Shen \cite{GS}.

Applying directly the results of Keller--Liverani \cite{KeLi}
would only bound the contributions
of the first and second terms of \eqref{dec0}  by $ |t-t_0|^{\eta}$ for $\eta <1/2$.
In order to estimate
the second term, we prove that $\widehat \LL_{t,M} -\widehat  \LL_{t_0,M}$ {\it acting
on the maximal eigenvector} is $O(|\log |t-t_0||^\Gamma |t-t_0|^{1/2})$
\emph{in the strong
\footnote{The strong norm plays here the role
of $C^{j}$ in  \eqref{LaYo}.} norm;}
in the Misiurewicz--Thurston case we get
get a better $O( |t-t_0|^{1/2})$ control).
It is usually not possible to obtain strong
norm bounds when bifurcations are present \cite{BY, KeLi}, and this remarkable feature here
is due to our choice of admissible pairs (combined with 
the fact that the towers for $f_t$ and $f_{t_0}$ are identical up to level $M$).  To estimate
the first term, we enhance the Keller--Liverani argument, using again
that it suffices to estimate the
 perturbation for the operators acting on the maximal eigenvector. 
 
 The changes just described are already needed to obtain the exponent
 $1/2$ in the upper bound~\eqref{main0}. 
 To get lower bound in~\eqref{main}, we use that the tower associated to a
 Misiurewicz--Thurston map $f_{t_0}$ can be required to have levels
 with sizes bounded from below, and that the truncation level
 can be chosen to be slightly larger. Finally, working
 with Banach norms based on $L^1$ as in \cite{BS2} would give that the
 first two terms  in \eqref{dec0} are $\le C  |t-t_0|^{1/2}$, while the third is 
 $\ge C ^{-1} |t-t_0|^{1/2}$ for some large constant $C>1$.
  However, \emph{introducing Banach--Sobolev norms based on $L^p$ 
 for $p>1$} instead, we are able  to control the constants and show that the last
 term dominates the other two.



\begin{thebibliography}{7}

\bibitem{ACF}
   Alves,  J.F.,  Carvalho, M.,
    Freitas, J.M.,
     Statistical stability for {H}\'enon maps of the {B}enedicks-{C}arleson type,
   \emph{{A}nn. {I}nst.    {H}. {P}oincar\'e  {A}nal. {N}on {L}in\'eaire}
 \textbf{27} (2010),
    595--637.
  
  
 \bibitem{ACF2}
    Alves,  J.F., Carvalho, M.,
    Freitas, J.M.,
    Statistical stability and continuity of {SRB} entropy for systems with {G}ibbs-{M}arkov structures,
   \emph{{C}omm. {M}ath. {P}hys.}  
\textbf{296} (2010),
739--767.

\bibitem{ALM} Avila, A., Lyubich, M.,   de Melo, W.,
Regular or stochastic dynamics in real analytic families of
unimodal maps,
\emph{Invent. Math.} \textbf{154} (2003), 451--550. 

\bibitem{B0}
Baladi, V.,
     \emph{Positive Transfer Operators and Decay of Correlations}.
     Advanced Series in Nonlinear Dynamics~16, World Scientific,
  Singapore, 2000.

 \bibitem{B1}
 Baladi, V.,
On the susceptibility function of piecewise expanding interval maps,
   \emph{Comm. Math. Phys.}
 \textbf{275} (2007),
  839--859.

\bibitem{B2}
Baladi, V., Linear response despite critical points,
\emph{Nonlinearity} \textbf{21} (2008), T81--T90.

\bibitem{BBS}
Baladi, V.,  Benedicks, M., Schnellmann, D.,
Whitney--H\"older continuity of the SRB measure for transversal families of smooth unimodal maps,
Preprint, arXiv:1302.3191  (2013).

\bibitem{BMS}
Baladi, V.,  Marmi, S.,   Sauzin, D.,
Natural boundary for the susceptibility function of generic
piecewise expanding unimodal maps,
\emph{Ergodic Theory Dynam. Systems} \textbf{34} 
(2014), 777--800.

\bibitem{BS1}
Baladi, V.,  Smania, D.,
 Linear response formula for piecewise expanding unimodal maps,
  \emph{Nonlinearity}
    \textbf{21}    (2008), 677--711.
     (Corrigendum, \emph{Nonlinearity} \textbf{25} (2012), 2203--2205.)

\bibitem{BS1'}
Baladi, V.,  Smania, D., Smooth deformations of piecewise expanding unimodal maps, 
\emph{DCDS Series A} \textbf{23} (2009), 685--703.


\bibitem{BS2'}
Baladi, V., Smania, D.,
    Analyticity of the {SRB} measure for holomorphic families of quadratic-like {C}ollet-{E}ckmann maps,
   \emph{Proc. Amer. Math. Soc.}
   \textbf{137} (2009),  1431--1437.

\bibitem{BS2}
Baladi, V., Smania, D.,
Linear response for smooth deformations of generic nonuniformly hyperbolic unimodal maps,
    \emph{Ann. Sci. \'Ec. Norm. Sup.}
    \textbf{45} (2012), 861--926.

\bibitem{BV} Baladi, V., Viana, M.,
    Strong stochastic stability and rate of mixing for unimodal
              maps, \emph{Ann. Sci. \'Ec. Norm. Sup.}  \textbf{29}
    (1996), 483--517.
 
\bibitem{BY} Baladi, V., Young, L.-S., 
On the spectra of randomly perturbed expanding maps,
\emph{Comm. Math. Phys.} \textbf{156} (1993), 
355--385.

\bibitem{Bo} E. Borel, \emph{Le\c cons sur les fonctions monog\`enes uniformes d'une variable complexe}. Gauthier-Villars, Paris, 1917.


\bibitem{BrSi} Breuer, J.,  Simon, B.,
Natural boundaries and spectral theory,
\emph{Adv. Math.} \textbf{226} (2011), 4902--4920.


\bibitem{BL}
Butterley, O.,   Liverani, C., 
 Smooth Anosov flows: correlation spectra and stability, \emph{J. Mod. Dyn.}
\textbf{1} (2007),  301--322.



\bibitem{dMvS} de Melo, W.,  van Strien, S., 
\emph{One-Dimensional Dynamics}.
Ergebnisse der Mathematik und ihrer Grenzgebiete, 
Springer-Verlag,  Berlin, 1993.

\bibitem{Dob} Dobbs, N., Todd, M.,
Entropy jumps up, Manuscript (2014)

\bibitem{Do}
 Dolgopyat, D.,
 On differentiability of SRB states for partially hyperbolic systems,
\emph{Invent. Math.} \textbf{155} (2004), 389--449.

\bibitem{Er}  Ershov, S.V., Is a perturbation theory for dynamical
chaos possible? \emph{Physics Letters A} \textbf{177} (1993), 180--185.

\bibitem{GS} Gao, B., Shen, W., Summability implies Collet--Eckmann almost surely,
\emph{Ergodic Theory Dynam. Systems,}
DOI: http://dx.doi.org/10.1017/etds.2012.173.
  

\bibitem{GL1}
Gou\"ezel, S., Liverani, C.,
Banach spaces adapted to Anosov systems,
\emph{Ergodic Theory  Dynam. Systems}
\textbf{26} (2006), 189--217.

\bibitem{GL2}
Gou\"ezel, S., Liverani, C., Compact locally maximal hyperbolic sets for smooth maps: fine statistical properties, \emph{J. Diff. Geom.} \textbf{79} (2008), 433--477.

\bibitem{HM} Hairer, M., Majda, A.J.,
A simple framework to justify linear response theory,
\emph{Nonlinearity} \textbf{23} (2010),  909--922.

\bibitem{Ji}
Jiang, M.,
Differentiating potential functions of SRB measures on hyperbolic attractors,
\emph{Ergodic Theory Dynam. Systems} \textbf{32} (2012), 1350--1369.

\bibitem{JL1}
 Jiang, M., de la Llave, R.,
Linear response function for coupled hyperbolic attractors, \emph{Comm. Math. Phys.} 
\textbf{261} (2006),  379--404. 

\bibitem{JL2} Jiang, M., de la Llave, R., 
Smooth dependence of thermodynamic limits of SRB-measures,
 \emph{Comm. Math. Phys.} 
\textbf{211} (2000), 303--333.

\bibitem{Kat}
Kato, T.,
\emph{Perturbation Theory for Linear Operators}.
Second corrected printing of the second
edition,  Springer-Verlag, Berlin, 1984.

\bibitem{KH} Katok, A., Hasselblatt, B., 
\emph{Introduction to the modern theory of dynamical systems}. 
Encyclopedia of Mathematics and its Applications, 54. Cambridge University Press, Cambridge, 1995.

\bibitem{KKPW} 
Katok, A., Knieper, G., Pollicott, M., Weiss, H.,
 Differentiability and analyticity of topological entropy for
Anosov and geodesic flows,  \emph{Invent. Math.}
    \textbf{98} (1989), 581--597.


\bibitem{KS1} Katok, A., Spatzier, R.J., 
First cohomology of Anosov actions of higher rank abelian
groups and applications to rigidity, \emph{Publ. Math., Inst. Hautes \'Etud. Sci.}
\textbf{79} (1994), 131--156.

\bibitem{KS2}
 Katok, A., Spatzier, R.J., Invariant measures for higher-rank hyperbolic abelian actions.
\emph{Ergodic Theory Dynam. Systems}
  \textbf{16} (1996), 751--778.
  
\bibitem{Ke}
Keller, G.,
   Stochastic stability in some dynamical systems,
   \emph{Monatshefte Math.}
    \textbf{94} (1982), 313--333.

\bibitem{KHK}
Keller, G., Howard, P.J., Klages, R., Continuity properties of transport coefficients in simple maps,
   \emph{Nonlinearity} \textbf{21} (2008), 1719--1743.

\bibitem{KeLi}
Keller, G., Liverani, C.,
	Stability of the spectrum for transfer operators,
   \emph{Ann. Scuola Norm. Sup. Pisa Cl. Sci.}
  \textbf{28} (1999), 141--152.
  
  \bibitem{keno}
   Keller, G., Nowicki, T.,
     Spectral theory, zeta functions and the distribution of periodic points for {C}ollet--{E}ckmann maps,   \emph{Comm. Math. Phys.}
     \textbf{149}
      (1992), {633--680}.

\bibitem{KS}
 Krzy\.{z}ewski, K., Szlenk, W., On invariant measures for expanding differentiable mappings,
\emph{ Studia Math.} \textbf{33} (1969), 83--92.

\bibitem{LY} Lasota, A., Yorke, J.A.,
On the existence of invariant measures for piecewise monotonic transformations,
\emph{Trans. Amer. Math. Soc.} \textbf{186} (1973), 481--488. 

\bibitem{LMM}
de la Llave, R., Marco, J.M., Moriy\'on, R.,
Canonical perturbation theory of Anosov systems and regularity results for the Liv\v sic cohomology equation,
\emph{Ann. of Math.} \textbf{123} (1986), 537--611.

\bibitem{MSS} Ma\~n\'e, R.,  Sad, P.,   Sullivan, D.,
On the dynamics of rational maps,
\emph{Ann. Sci. \'Ecole Norm. Sup.} \textbf{16} (1983), 193--217.


\bibitem{MM} Mazzolena, M., Dinamiche espansive unidimensionali:
dipendenza della misura invariante da un parametro, 
\emph{Master's Thesis} Roma 2, Tor Vergata (2007). 

\bibitem{NP}
Nowicki, T., Przytycki, F.,
Topological invariance of the Collet-Eckmann property for S-unimodal maps,
\emph{Fund. Math.} \textbf{155} (1998), 33--43.

\bibitem{R0}
Ruelle, D., Differentiation of SRB states,
\emph{Comm. Math. Phys.} \textbf{187} (1997), 227--241.

 

\bibitem{R1}
Ruelle, D.,
Differentiation of {SRB} states: Corrections and complements,
 \emph{Comm. Math. Phys.}
\textbf{234} (2003), 185--190.

\bibitem{R4}
    Ruelle, D.,
    General linear response formula in statistical mechanics, and the fluctuation-dissipation theorem far from equilibrium, \emph{Phys. Lett. A}  \textbf{245}     (1998), 220--224.
    
\bibitem{Ruru}    Ruelle, D.,
Nonequilibrium statistical mechanics near equilibrium: computing higher-order terms,
\emph{Nonlinearity} \textbf{11} (1998),  5--18.
  
\bibitem{R5}
    Ruelle, D.,
 Application of hyperbolic dynamics to physics: some problems and conjectures,
   \emph{Bull. Amer. Math. Soc.}
    \textbf{41} (2004), 275--278.

 
\bibitem{R2} Ruelle, D.,
Differentiating the absolutely continuous invariant measure of an interval map $f$ with respect to $f$,
   \emph{Comm. Math. Phys.}
    \textbf{258}
  (2005),  445--453.


\bibitem{R8} Ruelle, D.,
Differentiation of SRB states for hyperbolic flows,
\emph{Ergodic Theory Dynam. Systems} \textbf{28} (2008),  613--631.
	
\bibitem{R3}
Ruelle, D.,
 Structure and $f$-dependence of the A.C.I.M. for a unimodal map $f$  of Misiurewicz Type,
   \emph{Comm. Math. Phys.}
    \textbf{287} (2009), 1039--1070.



\bibitem{R6}
Ruelle, D.,
A review of linear response theory for general differentiable dynamical systems,
   \emph{Nonlinearity}
    \textbf{22} (2009), 855--870.

\bibitem{R7}
Ruelle, D.,
Singularities of the susceptibility of a Sinai-Ruelle-Bowen measure
in the presence of stable-unstable tangencies,
\emph{Philos. Trans. R. Soc. Lond. Ser. A Math. Phys. Eng. Sci.}
 \textbf{369} (2011), 482--493.

\bibitem{RySo} Rychlik, M.,  Sorets, E.,
    Regularity and other properties of absolutely continuous invariant measures for the quadratic family,
   \emph{Comm. Math. Phys.}
    \textbf{150} (1992), 217--236.
    
    \bibitem{ST} Sauzin, D., Tiozzo, G.,
    Generalised continuation by means of right limits,
    Preprint,  	arXiv:1301.1175 (2013). 

\bibitem{Sc} Schnellmann, D.,  Typical points for one-parameter families of piecewise expanding maps of the interval,
\emph{Discrete Contin. Dyn. Syst.} \textbf{31} (2011), 877--911. 


\bibitem{Sc'} Schnellmann, D., Law of iterated logarithm and invariance principle for one-parameter families of interval maps,
Preprint,  arXiv:1309.2116 (2013), to appear PTRF. 

 \bibitem{WX}
Shen, W.,
On stochastic stability of non\-uni\-form\-ly expanding interval maps,
\emph{Proc. Lond. Math. Soc.}  \textbf{107} (2013),  1091--1134.

\bibitem{thun}
  Thunberg, H., 
Unfolding of chaotic unimodal maps and parameter dependence of natural
measures,
  \emph{Nonlinearity}
\textbf{14}
  (2001),
323--338.

 \bibitem{Ts'}
Tsujii, M.,    Positive {L}yapunov exponents in families of one-dimensional
              dynamical systems, \emph{Invent. Math.}
 \textbf{111}    (1993), 113--137.


\bibitem{Ts}
Tsujii, M.,
     On continuity of {B}owen-{R}uelle-{S}inai measures in families of one-dimensional maps,
   \emph{Comm. Math. Phys.}
  \textbf{177} (1996), 1--11.

\bibitem{Wh}
Whitney, H.,
Analytic extensions of differentiable functions defined in closed sets,
\emph{Trans. Amer. Math. Soc.}
\textbf{36} (1934), 63--89.


\bibitem{young92}
Young, L.-S.,
     Decay of correlations for certain quadratic maps,
   \emph{Comm. Math. Phys.}
    \textbf{146}
 (1992),
   123--138.
     
\bibitem{YoYo} Young, L.-S.,
 Statistical properties of dynamical systems with some hyperbolicity. 
\emph{Ann. of Math.} \textbf{147} (1998),  585--650.


\bibitem{Yo} Young, L.-S.,
 What are SRB measures, and which dynamical systems have them? 
   \emph{J. Statist. Phys.}
    \textbf{108} (2002), 733--754.

\end{thebibliography}
\end{document}